\pdfoutput=1
\documentclass[letterpaper]{amsart}

\usepackage[utf8]{inputenc}
\usepackage{lmodern}
\usepackage[T1]{fontenc}
\usepackage{amssymb}
\usepackage{enumitem}
\usepackage{mathtools}

\newif\ifarxiv
\arxivtrue
\ifarxiv
\usepackage{tikz-cd}
\usepackage[pdfusetitle]{hyperref}
\else

\usepackage{tikz-cd}
\usepackage[dvipdfmx,pdfusetitle]{hyperref}
\fi
\usepackage{cleveref}

\tikzcdset{arrow style=Latin Modern}

\allowdisplaybreaks[4]

\newlist{enumarabic}{enumerate}{1}
\setlist[enumarabic]{font=\normalfont,label=(\arabic*),leftmargin=0.3in}
\newlist{enumroman}{enumerate}{1}
\setlist[enumroman]{font=\normalfont,label=(\roman*),leftmargin=0.3in}

\theoremstyle{plain}
\newtheorem{theorem}{Theorem}[section]
\newtheorem{proposition}[theorem]{Proposition}
\newtheorem{lemma}[theorem]{Lemma}
\newtheorem{corollary}[theorem]{Corollary}

\theoremstyle{definition}

\newtheorem{question}[theorem]{Question}

\theoremstyle{remark}
\newtheorem*{acknowledgements}{Acknowledgements}

\numberwithin{equation}{section}

\DeclareMathAlphabet\mathbfit{OML}{cmm}{b}{it}

\let\newterm\emph

\def\doi#1{\href{http://dx.doi.org/#1}{doi:#1}}

\def\cf{\emph{cf.}}

\let\epsilon\varepsilon
\let\phi\varphi
\let\emptyset\varnothing

\def\N{\mathbb{N}}
\def\Z{\mathbb{Z}}
\def\kk{\Bbbk}

\def\deg#1{|#1|}

\let\shuffle\nabla
\def\AW{AW}
\def\h{H}

\def\tsz{t_{\mathrm{sz}}}
\def\GG{\Omega}
\def\hh{\tilde\h}

\def\tX{\tilde{X}}
\def\tY{\tilde{Y}}
\def\tZ{\tilde{Z}}
\def\tpartial{\tilde{\partial}}
\def\ts{\tilde{s}}
\def\td{\tilde{d}}
\def\tiota{\tilde{T}}
\def\ttau{\tilde{\tau}}
\def\tT{\tilde{T}}
\def\talpha{\tilde{\alpha}}
\def\tbeta{\tilde{\beta}}

\def\id{\operatorname{id}}

\def\otimesL#1{\mathbin{{}_{#1}\otimes}}
\def\ii{\mathbfit{i}}

\begin{document}

\title[Szczarba's twisting cochain]{Szczarba's twisting cochain\\and the Eilenberg--Zilber maps}
\hypersetup{pdftitle={Szczarba's twisting cochain and the Eilenberg-Zilber maps}}
\author{Matthias Franz}
\thanks{The author was supported by an NSERC Discovery Grant.}
\address{Department of Mathematics, University of Western Ontario,
  London, Ont.\ N6A\;5B7, Canada}
\email{mfranz@uwo.ca}


\begin{abstract}
  We show that Szczarba's twisting cochain for a twisted Cartesian product
  is essentially the same as the one constructed by Shih.
  More precisely, Szczarba's twisting cochain can be obtained via the basic perturbation lemma
  if one uses a `reversed' version of the classical Ei\-len\-berg--Mac\,Lane
  homotopy for the Ei\-len\-berg--Zilber contraction.
  Along the way we prove several new identities involving these homotopies.
\end{abstract}

\maketitle

\section{Introduction}

Let \(F\hookrightarrow E\to B\) be a fibre bundle.
In 1959, E.~H.~Brown~\cite[Sec.~4]{Brown:1959} showed that for path-connected base~\(B\) the homology of the total space~\(E\)
is isomorphic to that of the twisted tensor product
\begin{equation}
  C(F)\otimes_{t} C(B).
\end{equation}
Here ``twisted'' means that the the usual tensor product differential is modified to a twisted differential~\(d_{t}\).
Brown used acyclic models to construct the twisting cochain~\(t\colon C(B)\to C(\Omega B)\) that determines this deviation.

Shortly afterwards, such twisting cochains were constructed by less opaque means
in the simplicial setting, that is, for twisted Cartesian products~\(F\times_{\tau}B\),
where \(F\) and~\(B\) are simplicial sets and \(\tau\colon B_{>0}\to G\) is a twisting function with values in the structure group~\(G\).
Szczarba~\cite{Szczarba:1961} gave an explicit formula for~\(t\) in terms of~\(\tau\)
while Shih~\cite{Shih:1962} described an algorithm for computing \(t\)
that starts with the Eilenberg--Zilber maps
(Alexander--Whitney map~\(\AW\), shuffle map~\(\shuffle\), Ei\-len\-berg--Mac\,Lane's homotopy~\(\h\))
and applies a perturbation determined by~\(\tau\).
This latter approach was the birth of what is nowadays called homological perturbation theory.

Later Rubio~\cite[p.~53]{Rubio:1991} observed experimentally that Shih's and Szczarba's twisting cochains agree
if one applies Shih's algorithm not to the usual Ei\-len\-berg--Zil\-ber contraction~\((\AW,\shuffle,\h)\), but to the ``opposite''
triple~\((\AW,\shuffle,\hh)\) where \(\hh\) is obtained from~\(\h\) by reversing simplices and transposing the factors
(see \Cref{sec:ez-reversed} for details). In \Cref{thm:szarba-shih} we prove this assertion.

\begin{theorem}
  \label{thm:main}
  Szczarba's twisting cochain is equal to Shih's twisting cochain
  based on the opposite Ei\-len\-berg--Zilber contraction~\((\AW,\shuffle,\hh)\).
\end{theorem}

The proof uses a characterization of Szczarba's twisting cochain obtained by Morace--Prouté~\cite{MoraceProute:1994}.

In order to apply Shih's theory to the opposite Eilenberg--Zilber contraction, we investigate
how both~\(\h\) and~\(\hh\) interact with the shuffle map~\(\shuffle\) and the Alexan\-der--Whitney map~\(\AW\).
This in fact occupies the largest part of this paper.
Such relations are all the more interesting as the homotopies~\(\h\) and~\(\hh\)
gives rise to important structures on cochain algebras: they
can be used to define Steenrod's \(\cup_{i}\)-products \cite[Thm.~3.2]{GonzalezDiazReal:1999}
as well as the homotopy Gerstenhaber structure \cite[Sec.~5]{HessEtAl:2006} 
that induces a product on the bar construction.

Besides offering rigorous proofs for certain commutative diagrams given by Shih as well as their ``mirror images''
(Propositions~\ref{thm:aw-h} and~\ref{thm:shuffle-h}),
we obtain several new identities. For example, the homotopies~\(\h_{X\times Y}\) and~\(\h_{X,Y\times Z}\)
that arise from the two ways of splitting up a triple Cartesian product~\(X\times Y\times Z\)
commute in the graded sense,
\begin{equation}
  \h_{X\times Y,Z}\,\h_{X,Y\times Z} = -\h_{X,Y\times Z}\,\h_{X\times Y,Z},
\end{equation}
see \Cref{thm:h-h-xyz}. As shown in \Cref{sec:ez-reversed},
all identities carry over to~\(\hh\).

We fix our notation in \Cref{sec:notation}.
The classical Eilenberg--Zilber contraction is studied in \Cref{sec:ez} and the opposite one in \Cref{sec:ez-reversed}.
In \Cref{sec:relations} we summarize all known relations among the Eilenberg--Zilber maps and ask if there are any further.
Szczarba's and Shih's twisting cochains are compared in \Cref{sec:tw-compare} and their twisted shuffle maps in \Cref{sec:twisted-shuffle}.

\begin{acknowledgements}
  I thank Francis Sergeraert for sending me Rubio's thesis~\cite{Rubio:1991}.
\end{acknowledgements}

\section{Preliminaries}
\label{sec:notation}

Let \(\kk\) be a commutative ring with unit. All complexes and tensor products we consider are over~\(\kk\).
The identity map of a complex is written as~\(1\).
Given two complexes~\(A\) and~\(B\), the transposition of factors is defined to be the chain map
\begin{equation}
  \label{eq:def-T}
  T_{A,B}\colon A\otimes B\to B\otimes A,
  \quad
  a\otimes b\mapsto (-1)^{\deg{a}\deg{b}}\,b\otimes a.
\end{equation}

Throughout this note, the letters~\(X\),~\(Y\),~\(Z\) and~\(W\) denote simplicial sets.
Let \(x\in X_{n}\) be a simplex. For~\(0\le p\le q\le n\) we write
\begin{equation}
  \partial_{p}^{q}\,x = \partial_{p}\cdots\partial_{q}\,x,
  \qquad
  \partial_{p}^{p-1}\,x = x
\end{equation}
for the iteration of face maps.
Given a set~\(\{\,\alpha_{1}<\dots<\alpha_{p}\,\}\) of non-negative integers with~\(\alpha_{p}\ge n+p-1\) if~\(p>0\), we also write
\begin{equation}
  s_{\alpha}\,x = s_{\alpha_{p}}\cdots s_{\alpha_{1}}\,x,
  \qquad
  s_{\emptyset}\,x = x
\end{equation}
for the iteration of degeneracy maps.
We denote the normalized chain complex of~\(X\) with coefficients in~\(\kk\) by~\(C(X)\).

Let \(p\),~\(q\ge0\).
A \((p,q)\)-shuffle is a partition of the set~\(\{\,0,\dots,p+q-1\,\}\) into subsets~\(\alpha\) and~\(\beta\)
of sizes~\(\deg{\alpha}=p\) and~\(\deg{\beta}=q\); it is denoted by~\((\alpha,\beta)\vdash(p,q)\).
We write \((-1)^{(\alpha,\beta)}\) for its signature, that is,
the sign of the permutation~\((\alpha_{1}<\dots<\alpha_{p},\allowbreak\beta_{1}<\dots<\beta_{q})\).

By a \newterm{simplicial operator} we mean what is called an ``FD-operator'' in~\cite[Sec.~3]{EilenbergMacLane:1953}.
The definition of the derived operator~\(f'\) of a simplicial operator~\(f\) appears in~\cite[p.~59]{EilenbergMacLane:1953},
for tensor products in~\cite[p.~53]{EilenbergMacLane:1954} and for the shuffle map in~\cite[eq.~(5.3\('\))]{EilenbergMacLane:1953}.
\newterm{Frontal} simplicial operators are defined in~\cite[p.~60]{EilenbergMacLane:1953};
by~\cite[Lemma~3.3]{EilenbergMacLane:1953} they satisfy
\begin{equation}
  \label{eq:frontal}
   s_{0}\,f = f'\,s_{0}.
\end{equation}

Note that a priori a simplicial operator~\(f\) is defined on non-normalized chain complexes only. Even if \(f\) descends
to normalized chains, this may fail for~\(f'\). (An example is the differential~\(d\).) Moreover, the second part
of the following observation shows that Ei\-len\-berg--Mac\,Lane's definition of derived operators does not behave well
with respect to the Koszul sign rule. Nevertheless, for ease of referencing we do not modify the definition.

\begin{lemma}
  \label{thm:derived}
  Let \(f\) and~\(g\) be simplicial operators.
  \begin{enumroman}
  \item \label{thm:derived-degenerate}
    If \(f\) and \(g\) agree modulo degenerate chains, then so do \(f'\) and~\(g'\).
  \item \label{thm:derived-tensor}
    One has \((f\otimes g)' = (-1)^{\deg{g}}\,f'\otimes g'\).
  \end{enumroman}
\end{lemma}

\begin{proof}
  The assumption in the first claim means that every monotonic operator appearing in~\(g-f\)
  has a leading degeneracy operator in its canonical form, compare~\cite[p.~59]{EilenbergMacLane:1953}.
  Hence the same holds for~\(g'-f'=(g-f)'\).
  The second claim follows directly from the definition of derived operators.
\end{proof}

\section{The Eilenberg--Zilber maps}
\label{sec:ez}

We start by reviewing the Ei\-len\-berg--Zilber maps\footnote{%
  This is a misnomer. All three maps were introduced by Ei\-len\-berg--Mac\,Lane
  \cite[eq.~(5.3)]{EilenbergMacLane:1953}, \cite[eqs.~(2.8),~(2.13)]{EilenbergMacLane:1954},
  with the obvious inspiration for the Alexander--Whitney map.}.
The Alexander--Whitney map is given by
\begin{align}
  \AW=\AW_{X,Y}\colon C(X\times Y) &\to C(X)\otimes C(Y), \\
  \notag
  (x,y) &\mapsto \sum_{k=0}^{n} \partial_{k+1}^{n}\,x \otimes \partial_{0}^{k-1}\,y \\
\shortintertext{for~\((x,y)\in X_{n}\times Y_{n}\). The shuffle map is defined as}
  \shuffle=\shuffle_{X,Y}\colon C(X)\otimes C(Y) &\to C(X\times Y), \\*
  \notag
  x \otimes y &\mapsto \sum_{(\alpha,\beta)\vdash(p,q)}
    (-1)^{(\alpha,\beta)}\,(s_{\beta}\,x,s_{\alpha}\,y)
  \end{align}
where \(p=\deg{x}\) and~\(q=\deg{y}\).
The homotopy
\begin{equation}
  \label{eq:def-h}
  \h=\h_{X,Y}\colon C(X\times Y) \to C(X\times Y),
\end{equation}
\begin{multline*}
  \qquad (x,y) \mapsto \sum_{0\le p+q < n}\:\sum_{(\alpha,\beta)\vdash(p+1,q)} (-1)^{m+1+(\alpha,\beta)} \\
  {} \cdot \bigl(s_{\beta+m}\,s_{m-1}\,\partial_{n-q+1}^{n}\,x,
  s_{\alpha+m}\,\partial_{m}^{n-q-1}\,y\bigr) \qquad
\end{multline*}
for~\((x,y)\in X_{n}\times Y_{n}\) with~\(m=n-p-q\) has been recursively defined by  Ei\-len\-berg--Mac\,Lane~\cite[eq.~(2.13)]{EilenbergMacLane:1954}
via
\begin{equation}
  \label{eq:def-h-rec}
  H_{0} = 0,
  \qquad
  H_{n} = -H_{n-1}' + F_{n-1}'\,s_{0}
\end{equation}
for~\(n>0\), where
\begin{equation}
  \label{eq:def-F}
  F 
  =\shuffle\,AW\colon C(X\times Y)\to C(X\times Y).
\end{equation}
The explicit formula given above is due to Rubio and Morace, see~\cite[Sec.~3.1]{Rubio:1991},~\cite[Sec.~6]{Real:2000}.\footnote{%
  Contrary to the claim made in~\cite[Thm.~6.2]{Real:2000},
  Ei\-len\-berg--Mac\,Lane's homotopy~\eqref{eq:def-h-rec} does \emph{not} satisfy \(d(\h) = 1-\shuffle\AW\).
  The sign exponent~\(n-p-q=1\) is erroneously not taken into account when the equations (65) and~(70) are compared in~\cite[p.~85]{Real:2000}.}
We will also use that \(\h\) is frontal \cite[p.~53]{EilenbergMacLane:1954}.
  
The Eilenberg--Zilber maps~\(\AW\),~\(\shuffle\) and~\(\h\) satisfy the properties
stated in~\cite[Thm.~2.1a]{EilenbergMacLane:1954} and~\cite[\S II.1]{Shih:1962},\footnote{%
  The proof of the identity~\(\h\h=0\) in~\cite[p.~26]{Shih:1962}
  implicitly uses an argument like \Cref{thm:derived}\,\ref{thm:derived-degenerate}.}
\begin{equation}
  \label{eq:ez-contraction}
  \AW\,\shuffle = 1,
  \quad
  \shuffle\AW = 1 + d(\h),
  \quad
  \h\,\shuffle = 0,
  \quad
  \AW\,\h = 0,
  \quad
  \h\,\h = 0.
\end{equation}
This means that the triple~%
\( 
  (\AW,\shuffle,\h)
\) 
forms a \newterm{contraction} in the sense of homological perturbation theory,
compare~
\cite{Brown:1965},~\cite[Sec.~3]{Gugenheim:1972}, \cite[Sec.~3]{Real:2000}.

\begin{proposition}
  \label{thm:aw-h}
  \tikzcdset{diagrams={ampersand replacement=\&,column sep=huge}}
  The following diagrams commute.
  \begin{gather}
    \tag{\(\mathrm{A}_{1}\)}
    \label{eq:aw-h}
    \begin{tikzcd}
      C(X\times Y\times Z) \arrow{d}{\h_{X\times Y,Z}} \arrow{r}{\AW_{X,Y\times Z}} \& C(X)\otimes C(Y\times Z)\arrow{d}{1\otimes \h_{Y,Z}} \\
      C(X\times Y\times Z) \arrow{r}{\AW_{X,Y\times Z}} \& C(X)\otimes C(Y\times Z)
    \end{tikzcd}
    \\
    \tag{\(\mathrm{A}_{2}\)}
    \label{eq:h-aw}
    \begin{tikzcd}
      C(X\times Y\times Z) \arrow{d}{\h_{X,Y\times Z}} \arrow{r}{\AW_{X\times Y,Z}} \& C(X\times Y)\otimes C(Z) \arrow{d}{\h_{X,Y}\otimes 1} \\
      C(X\times Y\times Z) \arrow{r}{\AW_{X\times Y,Z}} \& C(X\times Y)\otimes C(Z)
    \end{tikzcd}
  \end{gather}    
\end{proposition}

\begin{proposition}
  \label{thm:shuffle-h}
  \tikzcdset{diagrams={ampersand replacement=\&,column sep=huge}}
  The following diagrams commute.
  \begin{gather}
    \tag{\(\mathrm{B}_{1}\)}
    \label{eq:shuffle-h}
    \begin{tikzcd}
      C(X)\otimes C(Y\times Z) \arrow{d}{1\otimes \h_{Y\times Z}} \arrow{r}{\shuffle_{X,Y\times Z}} \& C(X\times Y\times Z) \arrow{d}{\h_{X\times Y,Z}} \\
      C(X)\otimes C(Y\times Z) \arrow{r}{\shuffle_{X,Y\times Z}} \& C(X\times Y\times Z)
    \end{tikzcd}
    \\
    \tag{\(\mathrm{B}_{2}\)}
    \label{eq:h-shuffle}
    \begin{tikzcd}
      C(X\times Y)\otimes C(Z) \arrow{d}{\h_{X\times Y}\otimes 1} \arrow{r}{\shuffle_{X\times Y,Z}} \& C(X\times Y\times Z) \arrow{d}{\h_{X,Y\times Z}} \\
      C(X\times Y)\otimes C(Z) \arrow{r}{\shuffle_{X\times Y,Z}} \& C(X\times Y\times Z)
    \end{tikzcd}
  \end{gather}
\end{proposition}

The diagrams~\eqref{eq:h-aw} and~\eqref{eq:shuffle-h} appear in Shih's work~\cite[\S II.4, Lemmes 3~\&~3\,bis]{Shih:1962}.
Because the arguments given there contain numerous inaccuracies, we prove them from scratch.
The diagrams~\eqref{eq:aw-h} and~\eqref{eq:h-shuffle} are contained in Gugenheim's paper~\cite[Lemma~4.0]{Gugenheim:1972}.
However, the homotopy~\(\h\) is not specified there, and instead of a proof the reader is referred to~\cite{Shih:1962}, where these diagrams cannot be found.\footnote{%
  Given that Gugenheim swapped the factors of the twisted tensor product compared to Shih,
  he might have been thinking of the homotopy~\(\hh\) defined in~\eqref{eq:def-hh} below, compare Propositions~\ref{thm:aw-hh} and~\ref{thm:shuffle-hh}.}

Instead of making Shih's arguments for the diagram~\eqref{eq:h-aw} rigorous,
we will present a different proof of \Cref{thm:aw-h}
which is based on the non-recursive definition~\eqref{eq:def-h} of~\(H\)
and can easily be adapted to~\eqref{eq:aw-h}.
We postpone that part to \Cref{sec:ez-reversed} because
the notation will be more transparent for the homotopy~\(\hh\) to be introduced there.
In order to prove \Cref{thm:shuffle-h} we need some preparations.

\begin{proof}[Proof of \Cref{thm:shuffle-h}]
  Let us verify that before normalization the diagram~\eqref{eq:shuffle-h} commutes modulo degenerate chains.
  We write ``\(\equiv\)'' if two chains differ by a degenerate chain. We closely follow Shih's strategy.

  Let \(a=x\otimes(y,z)\) with~\(x\in X_{p}\), \(y\in Y_{q}\) and~\(z\in Z_{q}\).
  If \(p=0\), then the shuffle map essentially reduces to the identity map, and
  for~\(q=0\) both homotopies vanish. We may therefore assume \(p>0\) and~\(q>0\). We proceed by induction on~\(p+q\).

  By the recursive definition~\eqref{eq:def-h-rec} of~\(\h\)
  we have \(\h_{X\times Y,Z}\shuffle_{X,Y\times Z}(a)=A+B\) with
  \begin{equation}
    A = -\h_{X\times Y,Z}'\,\shuffle_{X,Y\times Z}(a),
    \qquad
    B = F_{X\times Y,Z}'\,s_{0}\,\shuffle_{X,Y\times Z}(a).
  \end{equation}
  Using the identity~\cite[eq.~(5.7)]{EilenbergMacLane:1953}
  \begin{equation}
    \label{eq:shuffle-shuffle-deriv}
    \shuffle(x\otimes y) = \shuffle'(x\otimes s_{0}\,y) + (-1)^{p}\,\shuffle'(s_{0}\,x\otimes y)
  \end{equation}
  together with~\cite[Thm.~3.2]{EilenbergMacLane:1953}, we get
  \begin{align}
    A &= -\h'\,\shuffle'(x\otimes s_{0}(y,z)) - (-1)^{p}\,\h'\,\shuffle'(s_{0}\,x\otimes(y,z)) \\
    \notag &= -(\h\,\shuffle)'(x\otimes s_{0}(y,z)) - (-1)^{p}\,(\h\,\shuffle)'(s_{0}\,x\otimes(y,z)), \\
  \intertext{which by induction and \Cref{thm:derived}\,\ref{thm:derived-degenerate} is congruent to}
    \notag &\equiv -(\shuffle\,(1\otimes\h))'(x\otimes s_{0}(y,z)) - (-1)^{p}\,(\shuffle\,(1\otimes\h))'(s_{0}\,x\otimes(y,z)) \\
  \intertext{\Cref{thm:derived}\,\ref{thm:derived-tensor} and the identity~\eqref{eq:frontal} for the frontal operator~\(\h\) now imply}
    \notag &= \shuffle'\,(1\otimes\h)'(x\otimes s_{0}(y,z)) + (-1)^{p}\,\shuffle'\,(1\otimes\h)'(s_{0}\,x\otimes(y,z)) \\
    \notag &= (-1)^{p}\,\shuffle'(x\otimes \h'\,s_{0}(y,z)) - \shuffle'(s_{0}\,x\otimes\h'(y,z)) \\
    \notag &= (-1)^{p}\,\shuffle'(x\otimes s_{0}\,\h'(y,z)) - \shuffle'(s_{0}\,x\otimes\h'(y,z))
  \end{align}
  On the other hand, writing \(b=s_{0}\,x\otimes s_{0}(y,z)\) and
  taking the version
  \( 
    s_{0}\,\shuffle\,a = \shuffle'\,b
  \) 
  of the identity~\eqref{eq:frontal} for the frontal operator~\(\shuffle\)
  into account \cite[eq.~(5.5)]{EilenbergMacLane:1953}, we find
  \begin{align}
    B &= F_{X\times Y,Z}'\,\shuffle_{X,Y\times Z}'\,b = (\shuffle_{X\times Y,Z}\,\AW_{X\times Y,Z}\,\shuffle_{X,Y\times Z})'(b) \\
  \shortintertext{which by \cite[Lemme~II.4.2]{Shih:1962} (or diagram~\eqref{eq:shuffle-aw-1} below), \Cref{thm:derived}
    and the associativity of the shuffle map gives}
    \notag &\equiv \bigl(\shuffle_{X\times Y,Z}\,(\shuffle_{X,Y}\otimes 1)(1\otimes\AW_{Y,Z})\bigr)'(b) \\
    \notag &= \bigl(\shuffle_{X,Y\times Z}\,(1\otimes\shuffle_{Y,Z})(1\otimes\AW_{Y,Z})\bigr)'(b) \\
    \notag &= \shuffle_{X,Y\times Z}'\,(1\otimes F_{Y,Z})'(b) = \shuffle_{X,Y\times Z}'\,(1\otimes F_{Y,Z}')(b) \\
    \notag &= \shuffle_{X,Y\times Z}'\bigl(s_{0}\,x\otimes F_{Y,Z}'\,s_{0}(y,z)\bigr).
  \end{align}
  Combining \(A\) and~\(B\) and using again the formulas~\eqref{eq:def-h-rec} and~\eqref{eq:shuffle-shuffle-deriv}, we obtain
  \begin{align}
    \h\,\shuffle(a) &\equiv (-1)^{p}\,\shuffle'(x\otimes s_{0}\,\h'(y,z)) + \shuffle'(s_{0}\,x\otimes\h(y,z)) \\
    \notag &= (-1)^{p}\,\shuffle(x\otimes\h(y,z)) = \shuffle\,(1\otimes\h)(a),
  \end{align}
  as desired.
  
  The proof for the diagram~\eqref{eq:h-shuffle} is completely analogous.
  See the diagram~\eqref{eq:shuffle-aw-2} for the required mirror image of~\cite[Lemme~II.4.2]{Shih:1962}.
\end{proof}

\begin{corollary}
  \label{thm:f-h-xyz}
  The following identities hold between \(H\) and~\(F=\shuffle\AW\):
  \begin{align*}
    F_{X,Y\times Z}\,H_{X\times Y,Z} = H_{X\times Y,Z}\,F_{X,Y\times Z}, \\
    F_{X\times Y,Z}\,H_{X,Y\times Z} = H_{X,Y\times Z}\,F_{X\times Y,Z}.
  \end{align*}
\end{corollary}

\begin{proof}
  Combine the diagrams~\eqref{eq:aw-h}\,+\,\eqref{eq:shuffle-h} and \eqref{eq:h-aw}\,+\,\eqref{eq:h-shuffle}, respectively.
\end{proof}

The next result also seems to be new.

\begin{proposition}
  \label{thm:h-h-xyz}
  The operators~\(\h_{X\times Y,Z}\) and~\(\h_{X,Y\times Z}\) commute in the graded sense,
  \begin{equation*}
    \h_{X\times Y,Z}\,\h_{X,Y\times Z} = -\h_{X,Y\times Z}\,\h_{X\times Y,Z}.
  \end{equation*}
\end{proposition}

\begin{proof}
  We proceed by induction on the degree~\(n\) of the (non-normalized) argument. The case~\(n=0\) is trivial since~\(H_{0}=0\).

  Using the recursive definition~\eqref{eq:def-h-rec} as well as formula~\eqref{eq:frontal}, we have for~\(n>0\)
  \begin{align}
    & \h_{X\times Y,Z}\,\h_{X,Y\times Z} = - \h_{X\times Y,Z}'\,\h_{X,Y\times Z} + F_{X\times Y,Z}'\,s_{0}\,\h_{X,Y\times Z} \\
    \notag &\: = \h_{X\times Y,Z}'\,\h_{X,Y\times Z}' - \h_{X\times Y,Z}'\,F_{X,Y\times Z}'\,s_{0} + F_{X\times Y,Z}'\,\h_{X,Y\times Z}'\,s_{0} \\
    \notag &\: = (\h_{X\times Y,Z}\,\h_{X,Y\times Z})' - (\h_{X\times Y,Z}\,F_{X,Y\times Z})'\,s_{0} + (F_{X\times Y,Z}\,\h_{X,Y\times Z})'\,s_{0}.
  \end{align}
  Analogously, we get
  \begin{multline}
    \h_{X,Y\times Z}\,\h_{X\times Y,Z} \\*
    = (\h_{X,Y\times Z}\,\h_{X\times Y,Z})' - (\h_{X,Y\times Z}\,F_{X\times Y,Z})'\,s_{0} + (F_{X,Y\times Z}\,\h_{X\times Y,Z})'\,s_{0}.
  \end{multline}
  Using the induction hypothesis together with \Cref{thm:f-h-xyz} and
  \Cref{thm:derived}\,\ref{thm:derived-degenerate} completes the proof.
\end{proof}

\section{The opposite Eilenberg--Zilber contraction}
\label{sec:ez-reversed}

In order to obtain another contraction~\((\AW,\shuffle,\hh)\),
we reverse the ``front'' and ``back'' of each simplex in a simplicial set~\(X\).
(This idea appears already in~\cite[Sec.~3.3]{Rubio:1991}.)
More precisely, we define the \newterm{opposite simplicial set}~\(\tX\) to be equal to~\(X\) as a graded set, but with new face and degeneracy map
\begin{equation}
  \tpartial_{k}\,x = \partial_{n-k}\,x,
  \qquad
  \ts_{k}\,x = s_{n-k}\,x
\end{equation}
for~\(x\in\tX_{n}=X_{n}\).
The associated differential~\(\td\) on the graded module \(C(\tX)=C(X)\) is related to the original one by
\begin{equation}
  \label{eq:d-td}
  \td\,x = (-1)^{\deg{x}}\, d\,x.
\end{equation}

We introduce the maps
\begin{align}
  \tiota_{X}\colon C(X) &\to C(\tX), & x &\mapsto (-1)^{\nu(\deg{x})}\,x \\
  \tT_{X,Y}\colon C(X)\otimes C(Y) &\to C(\tY)\otimes C(\tX), & x\otimes y &\mapsto (-1)^{\nu(\deg{x}+\deg{y})}\,y\otimes x
\end{align}
where \(\nu\colon\Z\to\Z_{2}\) is defined by
\begin{equation}
  \nu(n) = \frac{n(n+1)}{2}
  = \begin{cases}
    0 & \text{if \(n\equiv 0,3\pmod 4\),}\\
    1 & \text{if \(n\equiv 1,2\pmod 4\).}\\
  \end{cases}
\end{equation}

\begin{lemma}
  Both~\(\tiota_{X}\) and~\(\tT_{X,Y}\) are chain maps, natural in~\(X\) and~\(Y\).
\end{lemma}

\begin{proof}
  This is a direct verification.
\end{proof}

We write
\begin{equation}
  \label{eq:def-tau}
  \tau_{X,Y}\colon X\times Y\to Y\times X
\end{equation}
for the canonical transposition and 
\begin{equation}
  \ttau_{X,Y} = \tiota_{Y\times X}\,(\tau_{X,Y})_{*} = (\tau_{\tX,\tY})_{*}\,\tiota_{X\times Y}\colon C(X\times Y)\to C(\tY\times\tX).
\end{equation}
Note that here we are using the canonical isomorphism~\(\widetilde{X\times Y}=\tX\times\tY\).

We will see that both the shuffle map~\(\shuffle\) and the Alexander--Whitney map~\(\AW\) are invariant
under the simplex reversal procedure just discussed, combined with the transposition of factors.
This does not hold for the Ei\-len\-berg--Mac\,Lane homotopy~\(\h\) because the two factors of~\(X\times Y\) are not treated
in a symmetric way. We therefore introduce the map
\begin{equation}
  \label{eq:def-hh}
  \hh=\hh_{X,Y}\colon C(X\times Y)\to C(X\times Y),
\end{equation}
\begin{equation*}
  \hh(x,y) = \sum_{0\le p+q < n}\:\sum_{(\alpha,\beta)\vdash(p,q+1)} (-1)^{p+q+(\alpha,\beta)}\,
  \bigl(s_{\beta}\,\partial_{p+1}^{p+q}\,x, s_{p+q+1}\,s_{\alpha}\,\partial_{0}^{p-1}\,y\bigr).
\end{equation*}
Note that the sign exponent is slightly simpler than in the formula~\eqref{eq:def-h} for~\(\h\).

To see how \(\hh\) is derived from~\(\h\),
we need to understand the effect of transposition and reversal on shuffles.
For a \((p,q)\)-shuffle~\((\alpha,\beta)\) we write
  \begin{equation*}
    (\talpha,\tbeta) = (m-\alpha,m-\beta) = \bigl(m-\alpha_{p},\dots,m-\alpha_{1},m-\beta_{q},\dots,m-\beta_{1}\bigr)
  \end{equation*}
for the shuffle obtained by reversing the sequence~\((0,\dots,p+q-1=m)\).
  
\begin{lemma}
  \label{thm:parity}
  Let \((\alpha,\beta)\) be a \((p,q)\)-shuffle, \(p\),~\(q\ge0\). Then
  \begin{equation*}
    (-1)^{(\beta,\alpha)} = (-1)^{(\talpha,\tbeta)} = (-1)^{(\alpha,\beta)+pq}.
  \end{equation*}
\end{lemma}

\begin{proof}
  Applying \(pq\)~transpositions transforms~\((\beta,\alpha)\) into~\((\alpha,\beta)\).
  For~\((\talpha,\tbeta)\) we can reverse the sequences \(\alpha\), \(\beta\) and the whole set, so that the sign exponent changes by
  \begin{equation}
    \textstyle{\frac{1}{2}}\,(p+q)(p+q-1)-\textstyle{\frac{1}{2}}\,p(p-1)-\textstyle{\frac{1}{2}}\,q(q-1) = pq. \qedhere
  \end{equation}
\end{proof}

\begin{lemma}
  \label{thm:shuffle-aw-h-hh}
  The following diagrams commute.
  \tikzcdset{diagrams={ampersand replacement=\&,column sep=large}}
  \begin{gather*}
    \begin{tikzcd}
      C(X\times Y) \arrow{d}{\ttau_{X,Y}} \arrow{r}{\AW_{X,Y}} \&  C(X)\otimes C(Y) \arrow{d}{\tT_{X,Y}} \\
      C(\tY\times\tX) \arrow{r}{\AW_{\tY,\tX}} \& C(\tY)\otimes C(\tX)
    \end{tikzcd}
    \\
    \begin{tikzcd}
      C(X)\otimes C(Y) \arrow{d}{\tT_{X,Y}} \arrow{r}{\shuffle_{X,Y}} \& C(X\times Y) \arrow{d}{\ttau_{X,Y}} \\
      C(\tY)\otimes C(\tX) \arrow{r}{\shuffle_{\tY,\tX}} \& C(\tY\times\tX)
    \end{tikzcd}
    \\
    \begin{tikzcd}
      C(X\times Y) \arrow{d}{\ttau_{X,Y}} \arrow{r}{\h_{X,Y}} \& C(X\times Y) \arrow{d}{\ttau_{X,Y}} \\
      C(\tY\times\tX) \arrow{r}{\hh_{\tY,\tX}} \& C(\tY\times\tX)
    \end{tikzcd}
  \end{gather*}
\end{lemma}

\begin{proof}
  This is a direct computation.
  For the parity of the shuffles appearing in the diagrams for~\(\shuffle\) and~\(\h\) one uses \Cref{thm:parity}.
  
  The additional sign~\((-1)^{n+1}\) in the definition~\eqref{eq:def-h} of~\(\h\) compared to~\eqref{eq:def-hh} 
  compensates for the signs introduced by the map~\(\ttau_{X,Y}\).
  This is analogous to the sign appearing in the identity~\eqref{eq:d-td} relating the two differentials.
  For~\(\hh\) the sign is \((-1)^{n+1}\) instead of~\((-1)^{n}\) because \(\h\) is of degree~\(+1\) while \(d\) is of degree~\(-1\).
\end{proof}

\begin{proposition}
  \label{thm:ez-contr-hh}
  The triple~\((\AW,\shuffle,\hh)\) is a contraction.
\end{proposition}

\begin{proof}
  Apply \Cref{thm:shuffle-aw-h-hh} to the contraction~\((\AW,\shuffle,\h)\).
\end{proof}

\begin{proposition}
  \tikzcdset{diagrams={ampersand replacement=\&,column sep=huge}}
  \label{thm:aw-hh}
  The following diagrams commute.
  \begin{gather*}
    \tag{\(\mathrm{\tilde A}_{1}\)}
    \label{eq:aw-hh}
    \begin{tikzcd}
      C(X\times Y\times Z) \arrow{d}{\hh_{X\times Y,Z}} \arrow{r}{\AW_{X,Y\times Z}} \& C(X)\otimes C(Y\times Z)\arrow{d}{1\otimes\hh_{Y,Z}} \\
      C(X\times Y\times Z) \arrow{r}{\AW_{X,Y\times Z}} \& C(X)\otimes C(Y\times Z)
    \end{tikzcd}
    \\
    \tag{\(\mathrm{\tilde A}_{2}\)}
    \label{eq:hh-aw}
    \begin{tikzcd}[column sep=huge]
      C(X\times Y\times Z) \arrow{d}{\hh_{X,Y\times Z}} \arrow{r}{\AW_{X\times Y,Z}} \& C(X\times Y)\otimes C(Z) \arrow{d}{\hh_{X,Y}\otimes 1} \\
      C(X\times Y\times Z) \arrow{r}{\AW_{X\times Y,Z}} \& C(X\times Y)\otimes C(Z)
    \end{tikzcd}    
  \end{gather*}
\end{proposition}

\begin{proof}
  We start with the second diagram and
  consider \((x,y,z)\in(X\times Y\times Z)_{n}\) for some~\(n\ge0\). Applying the definitions for non-normalized chains, we get
  \begin{align}
    \label{eq:pf-hh-1-AW}
    \MoveEqLeft (\hh_{X,Y}\otimes 1)\,\AW_{X\times Y,Z}(x,y,z) \\*
    \notag = & \sum_{k=1}^{n}\:\sum_{0\le p+q<k}\:\sum_{(\alpha,\beta)\vdash(p,q+1)} (-1)^{p+q+(\alpha,\beta)} \\*
    \notag & \qquad\qquad\qquad {}\cdot \bigl( s_{\beta}\,\partial_{p+1}^{p+q}\,\partial_{k+1}^{n}\,x,
    s_{p+q+1}\,s_{\alpha}\,\partial_{0}^{p-1}\,\partial_{k+1}^{n}\,y \bigr)
    \otimes \partial_{0}^{k-1}\,z \\
  \intertext{(there are no terms for~\(k=0\)) and}
    \label{eq:pf-AW-hh-bis}
    \MoveEqLeft \AW_{X\times Y,Z}\,\hh_{X,Y\times Z}(x,y,z) \\*
    \notag = & \sum_{k=0}^{n+1}\:\sum_{0\le p+q<n}\:\sum_{(\alpha,\beta)\vdash(p,q+1)} (-1)^{p+q+(\alpha,\beta)} \\*
    \notag & \quad {}\cdot \bigl( \partial_{k+1}^{n+1}\,s_{\beta}\,\partial_{p+1}^{p+q}\,x,
    \partial_{k+1}^{n+1}\,s_{p+q+1}\,s_{\alpha}\,\partial_{0}^{p-1}\,y \bigr)
    \otimes \partial_{0}^{k-1}\,s_{p+q+1}\,s_{\alpha}\,\partial_{0}^{p-1}\,z.
  \end{align}
  The second tensor factor in the last line is degenerate if \(k-1<p+q+1\). 
  We may therefore assume \(k>p+q+1\ge 1\), in which case this term can be written as
  \begin{align}
    \partial_{0}^{k-1}\,s_{p+q+1}\,s_{\alpha}\,\partial_{0}^{p-1}\,z
    &= \partial_{0}^{k-p-2}\,\partial_{0}^{p-1}\,z
    = \partial_{0}^{k-2}\,z. \\
  \intertext{The components of the first tensor factor in~\eqref{eq:pf-AW-hh-bis} can likewise be transformed to}
    \partial_{k+1}^{n+1}\,s_{\beta}\,\partial_{p+1}^{p+q}\,x
    &= s_{\beta}\,\partial_{k-q}^{n-q}\,\partial_{p+1}^{p+q}\,x
    = s_{\beta}\,\partial_{p+1}^{p+q}\,\partial_{k}^{n}\,x, \\
    \partial_{k+1}^{n+1}\,s_{p+q+1}\,s_{\alpha}\,\partial_{0}^{p-1}\,y
    &= s_{p+q+1}\,s_{\alpha}\,\partial_{k-p}^{n-p}\,\partial_{0}^{p-1}\,y
    = s_{p+q+1}\,s_{\alpha}\,\partial_{0}^{p-1}\,\partial_{k}^{n}\,y.
  \end{align}
  Modulo degenerate chains this gives
  \begin{align}
    \MoveEqLeft \AW_{X\times Y,Z}\,\hh_{X,Y\times Z}(x,y,z) \\*
    \notag \equiv & \sum_{0\le p+q<n}\:\sum_{k=p+q+2}^{n+1}\:\sum_{(\alpha,\beta)\vdash(p,q+1)} (-1)^{p+q+(\alpha,\beta)} \\*
    \notag & \qquad\qquad\qquad\qquad {}\cdot \bigl( s_{\beta}\,\partial_{p+1}^{p+q}\,\partial_{k}^{n}\,x,
    s_{p+q+1}\,s_{\alpha}\,\partial_{0}^{p-1}\,\partial_{k}^{n}\,y \bigr)
    \otimes \partial_{0}^{k-2}\,z.
  \end{align}
  This is the same as~\eqref{eq:pf-hh-1-AW} after substituting \(k+1\) for~\(k\).

  We now turn to the first diagram, where the argument is similar. We have
  \begin{align}
    \label{eq:pf-1-hh-AW}
    \MoveEqLeft (1\otimes \hh_{Y,Z})\,AW_{X,Y\times Z}(x,y,z) \\*
    \notag = & \sum_{k=0}^{n}\:\sum_{0\le p+q<n-k}\:\sum_{(\alpha,\beta)\vdash(p,q+1)}(-1)^{k+p+q+(\alpha,\beta)} \\*
    \notag & \qquad\qquad\qquad {}\cdot \partial_{k+1}^{n}\,x \otimes
    \bigl( s_{\beta}\,\partial_{p+1}^{p+q}\,\partial_{0}^{k-1}\,y,
    s_{p+q+1}\,s_{\alpha}\,\partial_{0}^{p-1}\,\partial_{0}^{k-1}\,z \bigr) \\
    \notag = & \sum_{0\le k+p+q<n}\:\sum_{(\alpha,\beta)\vdash(p,q+1)}(-1)^{k+p+q+(\alpha,\beta)} \\*
    \notag & \qquad\qquad\qquad {}\cdot \partial_{k+1}^{n}\,x \otimes
    \bigl( s_{\beta}\,\partial_{p+1}^{p+q}\,\partial_{0}^{k-1}\,y,
    s_{p+q+1}\,s_{\alpha}\,\partial_{0}^{k+p-1}\,z \bigr)
  \end{align}
  and
  \begin{align}
    \label{eq:pf-AW-hh}
    \MoveEqLeft \AW_{X,Y\times Z}\,\hh_{X\times Y, Z}(x,y,z) \\*
    \notag = & \sum_{k=0}^{n+1}\sum_{0\le p+q<n}\sum_{(\alpha,\beta)\vdash(p,q+1)}(-1)^{p+q+(\alpha,\beta)} \\*
    \notag & \qquad\qquad {}\cdot \partial_{k+1}^{n+1}\,s_{\beta}\,\partial_{p+1}^{p+q}\,x \otimes
    \bigl( \partial_{0}^{k-1}\,s_{\beta}\,\partial_{p+1}^{p+q}\,y,
    \partial_{0}^{k-1}\,s_{p+q+1}\,s_{\alpha}\,\partial_{0}^{p-1}\,z \bigr).
  \end{align}
  The first tensor factor in~\eqref{eq:pf-AW-hh} is degenerate if \(\beta\) contains a value~\(< k\).
  We can therefore assume all such values to occur in~\(\alpha\), so that in particular we obtain \(k\le p<n\).
  Since we also have \(n-q\ge p+1\), the first tensor factor simplifies to
  \begin{align}
    \partial_{k+1}^{n+1}\,s_{\beta}\,\partial_{p+1}^{p+q}\,x
    &= \partial_{k+1}^{n-q}\,\partial_{p+1}^{p+q}\,x = \partial_{k+1}^{n}\,x
  \intertext{in this case.
  Let us write \(\hat p=p-k\) as well as \(\hat \alpha = \{\,i-k \mid k\le i\in\alpha\,\}\) and~\(\hat\beta=\beta-k\).
  The first component of the second tensor factor in~\eqref{eq:pf-AW-hh} can be expressed as}
    \partial_{0}^{k-1}\,s_{\beta}\,\partial_{p+1}^{p+q}\,y
    &= s_{\hat\beta}\,\partial_{0}^{k-1}\,\partial_{p+1}^{p+q}\,y
    = s_{\hat\beta}\,\partial_{\hat p+1}^{\hat p+q}\,\partial_{0}^{k-1}\,y
  \intertext{and the second component as}
    \partial_{0}^{k-1}\,s_{p+q+1}\,s_{\alpha}\,\partial_{0}^{p-1}\,z
    &= s_{\hat p+q+1}\,\partial_{0}^{k-1}\,s_{\alpha}\,\partial_{0}^{p-1}\,z
    = s_{\hat p+q+1}\,s_{\hat\alpha}\,\partial_{0}^{k+\hat p-1}\,z.
  \end{align}
  Modulo degenerate chains we therefore have
  \begin{align}
    \MoveEqLeft \AW_{X,Y\times Z}\,\hh_{X\times Y, Z}(x,y,z) \\*
    \notag \equiv & \sum_{k=0}^{p}\:\sum_{0\le p+q<n}\:\sum_{\substack{(\alpha,\beta)\vdash(p,q+1)\\\{0,\dots,k-1\}\subset\alpha}}(-1)^{p+q+(\alpha,\beta)} \\*
    \notag & \qquad\qquad\qquad\quad {}\cdot \partial_{k+1}^{n}\,x \otimes
    \bigl( s_{\hat\beta}\,\partial_{\hat p+1}^{\hat p+q}\,\partial_{0}^{k-1}\,y,
    s_{\hat p+q+1}\,s_{\hat\alpha}\,\partial_{0}^{k+\hat p-1}\,z \bigr) \\
    \notag = & \sum_{0\le k+\hat p+q<n}\:\sum_{(\hat\alpha,\hat\beta)\vdash(\hat p,q+1)}(-1)^{k+\hat p+q+(\hat\alpha,\hat\beta)} \\*
    \notag & \qquad\qquad\qquad\quad {}\cdot \partial_{k+1}^{n}\,x \otimes
    \bigl( s_{\hat\beta}\,\partial_{\hat p+1}^{\hat p+q}\,\partial_{0}^{k-1}\,y,
    s_{\hat p+q+1}\,s_{\hat\alpha}\,\partial_{0}^{k+\hat p-1}\,z \bigr),
  \end{align}
  which matches \eqref{eq:pf-1-hh-AW}.
\end{proof}

\begin{proof}[Proof of \Cref{thm:aw-h}]
  The top of the cube
  \begin{equation}
    \label{eq:cube-hh-aw}
    \begin{tikzcd}[column sep=-0.1em,row sep=2.5em]
      C(X\times Y\times Z) \arrow{dd} \arrow{rd}{\h_{X,Y\times X}} \arrow{rr}{\AW_{X\times Y,Z}} & & C(X\times Y)\otimes C(Z) \arrow{dd} \arrow{rd}{\h_{X,Y}\otimes 1} \\
      & C(X\times Y\times Z) \arrow[crossing over]{rr} & & C(X\times Y)\otimes C(Z) \arrow{dd} \\
      C(\tZ\times \tY\times \tX) \arrow{rr} \arrow{rd}[left]{\hh_{\tZ\times \tY,\tX}} & & C(\tZ)\otimes C(\tY\times \tX) \arrow{rd}{1\otimes\hh_{\tY,\tX}} \\
      & C(\tZ\times \tY\times \tX) \arrow[from=uu,crossing over] \arrow{rr}{\AW_{\tZ,\tY\times \tX}} & & C(\tZ)\otimes C(\tY\times \tX) 
    \end{tikzcd}
  \end{equation}
  is the commutative diagram~\eqref{eq:h-aw}, and the bottom is a relabelling of~\eqref{eq:aw-hh}.
  The front and back consist of the diagram
  \begin{equation}
    \begin{tikzcd}[column sep=huge]
      C(X\times Y\times Z) \arrow{d}[left]{\ttau_{X\times Y,Z}} \arrow{r}{\AW_{X\times Y,Z}} & C(X\times Y)\otimes C(Z) \arrow{d}{\tT_{X\times Y,Z}} \\
      C(\tZ\times\tX\times\tY) \arrow{d}[left]{(\id_{\tZ},\tau_{\tX,\tY})_{*}} \arrow{r}{\AW_{\tZ,\tX\times\tY}} & C(\tZ)\otimes C(\tX\times\tY) \arrow{d}{1\otimes(\tau_{\tX,\tY})_{*}} \\
      C(\tZ\times\tY\times\tX) \arrow{r}{\AW_{\tZ,\tY\times\tX}} & C(\tZ)\otimes C(\tY\times\tX) \mathrlap{.}
    \end{tikzcd}
  \end{equation}
  The left side is the diagram
  \begin{equation}
    \begin{tikzcd}[column sep=huge]
      C(X\times Y\times Z) \arrow{d}[left]{\ttau_{X,Y\times Z}} \arrow{r}{\h_{X\times Y,Z}} & C(X\times Y\times Z) \arrow{d}{\ttau_{X,Y\times Z}} \\
      C(\tY\times\tZ\times\tX) \arrow{d}[left]{(\tau_{\tY,\tZ},\id_{Z})_{*}} \arrow{r}{\hh_{\tY\times\tZ,\tX}} & C(\tY\times\tZ\times\tX) \arrow{d}{(\tau_{\tY,\tZ},\id_{Z})_{*}} \\
      C(\tZ\times \tY\times \tX) \arrow{r}{\hh_{\tZ\times\tY,\tX}} & C(\tZ\times \tY\times \tX) \mathrlap{,}
    \end{tikzcd}
  \end{equation}
  and the right side is
  \begin{equation}
    \begin{tikzcd}[column sep=huge]
      C(X\times Y)\otimes C(Z) \arrow{d}[left]{\ttau_{X,Y}\otimes\tiota_{Z}} \arrow{r}{\h_{X,Y}\otimes1} & C(X\times Y)\otimes C(Z) \arrow{d}{\ttau_{X,Y}\otimes\tiota_{Z}} \\
      C(\tY\times\tX)\otimes C(\tZ) \arrow{d}[left]{T_{C(\tY\times\tX),C(\tZ)}} \arrow{r}{\hh_{\tY,\tX}\otimes1} & C(\tY\times\tX)\otimes C(\tZ) \arrow{d}{T_{C(\tY\times\tX),C(\tZ)}} \\
      C(\tZ)\otimes C(\tY\times\tX) \arrow{r}{1\otimes\hh_{\tY,\tX}} & C(\tZ)\otimes C(\tY\times\tX) \mathrlap{.}
    \end{tikzcd}
  \end{equation}
  All four sides commute by \Cref{thm:shuffle-aw-h-hh} and naturality.
  One verifies directly that the composed vertical maps of the various diagrams agree wherever needed.
  To see that the signs work out correctly, one uses the identity
  \begin{equation}
    \nu(m+n) \equiv \nu(m) + \nu(n) +mn \pmod{2}
  \end{equation}
  for~\(m\),~\(n\ge0\). This defines the vertical arrows in~\eqref{eq:cube-hh-aw} and completes the argument.
  
  For~\eqref{eq:aw-h} one performs a similar diagram chase,
  this time starting with the commutative diagram~\eqref{eq:hh-aw}.
\end{proof}

\begin{proposition}
  \tikzcdset{diagrams={ampersand replacement=\&,column sep=huge}}
  \label{thm:shuffle-hh}
  The following diagrams commute.
  \begin{gather*}
    \tag{\(\mathrm{\tilde B}_{1}\)}
    \label{eq:shuffle-hh}
    \begin{tikzcd}[column sep=huge]
      C(X)\otimes C(Y\times Z) \arrow{d}{1\otimes\hh_{Y\times Z}} \arrow{r}{\shuffle_{X,Y\times Z}} \& C(X\times Y\times Z) \arrow{d}{\hh_{X\times Y,Z}} \\
      C(X)\otimes C(Y\times Z) \arrow{r}{\shuffle_{X,Y\times Z}} \& C(X\times Y\times Z)
    \end{tikzcd}
    \\
    \tag{\(\mathrm{\tilde B}_{2}\)}
    \label{eq:hh-shuffle}
    \begin{tikzcd}
      C(X\times Y)\otimes C(Z) \arrow{d}{\hh_{X\times Y}\otimes 1} \arrow{r}{\shuffle_{X\times Y,Z}} \& C(X\times Y\times Z) \arrow{d}{\hh_{X,Y\times Z}} \\
      C(X\times Y)\otimes C(Z) \arrow{r}{\shuffle_{X\times Y,Z}} \& C(X\times Y\times Z)
    \end{tikzcd}
  \end{gather*}
\end{proposition}

\begin{proof}
  This uses \Cref{thm:shuffle-h} and is otherwise
  analogous to the proof of \Cref{thm:aw-h}, done in reverse.
\end{proof}

\begin{corollary}
  \label{thm:f-hh-xyz-hh-hh-xyz}
  \Cref{thm:f-h-xyz} and \Cref{thm:h-h-xyz}
  remain valid for \(\hh\) instead of~\(\h\).
\end{corollary}

\begin{proof}
  The proof of \Cref{thm:f-h-xyz} carries over.
  For \Cref{thm:h-h-xyz} we note the equality of isomorphisms
  \begin{multline}
    \label{eq:ttau-tau}
    \qquad\quad
    \ttau_{X,Z\times Y}\,(\id,\tau_{Y,Z})_{*} = \ttau_{Y\times X,Z}\,(\tau_{X,Y},\id)_{*}\colon \\
    C(X\times Y\times Z)\to C(\tZ\times\tY\times\tX)
    \qquad\quad
  \end{multline}
  and write \(\sigma=(\tau_{X,Y},1)\,(\id,\tau_{Y,Z})^{-1}\colon X\times Z\times Y\to Y\times X\times Z\). The diagram
  \begin{equation}
    \kern-4ex
    \begin{tikzcd}[column sep=4.25em]
      C(X\times Y\times Z) \arrow{d}{\h_{X,Y\times Z}} \arrow{r}{(\id,\tau_{Y,Z})_{*}} & C(X\times Z\times Y) \arrow{d}{\h_{X,Z\times Y}} \arrow{r}{\ttau_{X,Z\times Y}} & C(\tZ\times\tY\times\tX) \arrow{d}{\h_{\tZ\times\tY,\tX}} \\
      C(X\times Y\times Z) \arrow{d}{=} \arrow{r}{(\id,\tau_{Y,Z})_{*}} & C(X\times Z\times Y) \arrow{d}{\sigma_{*}} \arrow{r}{\ttau_{X,Z\times Y}} & C(\tZ\times\tY\times\tX) \arrow{d}{=} \\
      C(X\times Y\times Z) \arrow{d}{\h_{X\times Y,Z}} \arrow{r}{(\tau_{X,Y},\id)_{*}} & C(Y\times X\times Z) \arrow{d}{\h_{Y\times X,Z}} \arrow{r}{\ttau_{Y\times X,Z}} & C(\tZ\times\tY\times\tX) \arrow{d}{\h_{\tZ,\tY\times\tX}} \\
      C(X\times Y\times Z) \arrow{r}{(\tau_{X,Y},\id)_{*}} & C(Y\times X\times Z) \arrow{r}{\ttau_{Y\times X,Z}} & C(\tZ\times\tY\times\tX)
    \end{tikzcd}
  \end{equation}
  commutes by naturality and \Cref{thm:shuffle-aw-h-hh} as well as the identity~\eqref{eq:ttau-tau} for the centre-right square.
  The same applies to the analogous diagram for the composition~\(\h_{X,Y\times Z}\,\h_{X\times Y,Z}\).
  The isomorphism~\eqref{eq:ttau-tau} therefore translates between \Cref{thm:h-h-xyz} and its analogue for~\(\hh\).
\end{proof}

\section{Relations among the Eilenberg--Zilber maps}
\label{sec:relations}

Let us summarize what is known about relations between the Eilenberg--Zilber maps and their interactions with the differential and
the transposition maps~\eqref{eq:def-T} and~\eqref{eq:def-tau}.
We focus on Eilenberg--Mac\,Lane's original homotopy~\(\h\).
For a one-point space~\(*\), we make the canonical identifications~\(C(*)=\kk\) and \(*\times X=X\times*=X\). 

\begin{enumroman}
\item
  The Alexander--Whitney map~\(\AW\), the shuffle map~\(\shuffle\) and the homotopy~\(\h\) are natural with respect to pairs of maps between simplicial sets.
\item If one argument is a one-point space, then \(\AW\) and~\(\shuffle\) reduce to the identity map on the chain complex of the other space,
  and \(\h\) reduces to~\(0\).
\item Both~\(\AW\) and~\(\shuffle\) are chain maps, and the contraction identities~\eqref{eq:ez-contraction} hold.
\item The Alexander--Whitney map is coassociative.
  The shuffle map is associative \cite[Thms.~5.2~\&~5.4]{EilenbergMacLane:1953}
  and also commutative in the sense that the following diagram commutes:
  \begin{equation}
    \begin{tikzcd}
      C(X)\otimes C(Y) \arrow{d}[left]{T_{C(X),C(Y)}} \arrow{r}{\shuffle_{X,Y}} & C(X\times Y) \arrow{d}{(\tau_{X,Y})_{*}} \\
      C(Y)\otimes C(X) \arrow{r}{\shuffle_{Y,X}} & C(Y\times X)
    \end{tikzcd}
  \end{equation}
\item
  \label{ez-prop-5}
  The diagram
  \begin{equation}
    \begin{tikzcd}[column sep=huge]
    C(X\times Y)\otimes C(Z\times W) \arrow{d}[left]{\AW_{X,Y}\otimes\AW_{Z,W}} \arrow{r}{\shuffle_{X\times Y,Z\times W}} & C(X\times Y\times Z\times W)\arrow{d}{(\id,\tau_{Y,Z},\id)_*} \\
    C(X)\otimes C(Y)\otimes C(Z)\otimes C(W) \arrow{d}[left]{1\otimes T_{C(Y),C(Z)}\otimes1} & C(X\times Z\times Y\times W) \arrow{d}{\AW_{X\times Z,Y\times W}} \\
    C(X)\otimes C(Z)\otimes C(Y)\otimes C(W) \arrow{r}{\shuffle_{X,Z}\otimes \shuffle_{Y,W}} & C(X\times Z)\otimes C(Y\times W)      
    \end{tikzcd}
  \end{equation}
  commutes \cite[Prop.~2.2.1]{Franz:2001}.
  Like the properties mentioned in the previous item, this holds already before normalization.
\item
  \label{ez-prop-6}
  One has the identities given in Propositions~\ref{thm:aw-h},~\ref{thm:shuffle-h} and~\ref{thm:h-h-xyz}.
\end{enumroman}

By setting one simplicial set equal to~\(*\),
one can deduce from~\ref{ez-prop-5} the commutative diagram~\cite[Lemme~II.4.2]{Shih:1962}
\begin{equation}
  \label{eq:shuffle-aw-1}
  \begin{tikzcd}[column sep=large]
    C(X)\otimes C(Y\times Z) \arrow{d}[left]{1\otimes\AW_{Y,Z}} \arrow{r}{\shuffle_{X,Y\times Z}} & C(X\times Y\times Z) \arrow{d}{\AW_{X\times Y,Z}} \\
    C(X)\otimes C(Y)\times C(Z) \arrow{r}{\shuffle_{X,Y}\otimes 1} & C(X\times Y)\otimes C(Z)
  \end{tikzcd}
\end{equation}
as well as its mirror image
\begin{equation}
  \label{eq:shuffle-aw-2}
  \begin{tikzcd}[column sep=large]
    C(X\times Y)\otimes C(Z) \arrow{d}[left]{\AW_{X,Y}\otimes1} \arrow{r}{\shuffle_{X\times Y,Z}} & C(X\times Y\times Z) \arrow{d}{\AW_{X,Y\times Z}} \\
    C(X)\otimes C(Y)\times C(Z) \arrow{r}{1\otimes\shuffle_{Y,Z}} & C(X)\otimes C(Y\times Z) \mathrlap{.}
  \end{tikzcd}
\end{equation}
Property~\ref{ez-prop-5} together with naturality also implies that the shuffle map is a morphism of coalgebras \cite[(17.6)]{EilenbergMoore:1966};
one can equivalently derive \ref{ez-prop-5} from this latter fact and naturality.
We have seen, moreover, that \Cref{thm:f-h-xyz} follows from property~\ref{ez-prop-6}.

\begin{question}
  Are there any relations involving the Eilenberg--Zilber maps
  that are not formal consequences of the properties listed above?
\end{question}

\section{Comparing the twisting cochains}
\label{sec:tw-compare}

We finally turn to twisted Cartesian products and compare the twisting cochains defined by Szczarba and Shih.

\subsection{Twisted Cartesian products}

Let \(B\) be a simplicial set, \(G\) a simplicial group and
\(\tau\colon B_{>0}\to G\) a twisting function. We use the convention
\begin{equation}
  \partial_{0}\,\tau(b) = \tau(\partial_{0}b)^{-1}\,\tau(\partial_{1}b)
\end{equation}
from~\cite[eq.~(1.1)]{Szczarba:1961},~\cite[Def.~18.3]{May:1968},
which corresponds to a left action of~\(G\) on bundle fibres.
To turn this into a right action one instead takes \(\sigma(b)=\tau(b)^{-1}\),
which is a twisting function in the sense of~\cite[p.~25]{Shih:1962}.

Let \(F\) be a simplicial set with a \(G\)-action. The twisted Cartesian product~\(F\times_{\tau}B\)
has the modified first face map
\begin{equation}
  \partial^{\tau}_{0}(f,b) = \bigl(\tau(b)\cdot\partial_{0}f,\partial_{0}b\bigr) = \bigl(\partial_{0}f\cdot\sigma(b),\partial_{0}b\bigr)
\end{equation}
for simplices of positive dimension, depending on whether one lets \(G\) act on the left or on the right.

\subsection{Shih's twisting cochain}

We recall the definition of Shih's twisting cochain, compare~\cite[\S II.1]{Shih:1962},~\cite{Brown:1965},~\cite[Sec.~4]{Gugenheim:1972}.
We write \(d_{\otimes}\) for the tensor product differential on~\(C(F)\otimes C(B)\) and
\begin{equation}
  \delta=d^{\tau}-d=
  \begin{cases}
    \partial^{\tau}_{0}-\partial_{0} & \text{in positive degree,} \\
    0 & \text{in degree~\(0\)}
  \end{cases}
\end{equation}
for the difference between the twisted and the untwisted differential on the graded module~\(C(F\times B)=C(F\times_{\tau}B)\).

\begin{lemma}
  \label{thm:delta-h-iter}
  Let \(f\) be a simplicial operator from the product~\(F\times B\) to itself.
  Then for any~\(c\in C(F\times B)\) there is an~\(n\ge0\) such that \((\delta\,f)^{n}(c)=0\).
\end{lemma}

\begin{proof}
  This follows from the naturality of~\(f\) with respect to the filtration of~\(F\times B\) by the skeletons of~\(B\),
  see~\cite[p.~34]{Brown:1965}.
\end{proof}

As a consequence,
the twisted differential~\(d^{\tau}\) can be transferred along the contraction~\((\AW,\shuffle,\hh)\) 
to the differential
\begin{equation}
  \label{eq:d-twisted-shih}
  d_{t} = d_{\otimes} + \sum_{n=0}^{\infty} \AW\,(\delta\,\hh)^{n}\,\delta\,\shuffle
\end{equation}
on~\(C(F)\otimes C(B)\). (\Cref{thm:delta-h-iter} guarantees that for any argument the sum is finite.)
Because \(\delta\) involves only the first face map and
the diagrams~\eqref{eq:hh-aw} and~\eqref{eq:shuffle-hh} commute, 
we obtain a twisted tensor product \(C(F)\otimes_{t} C(B)\)
related to~\(C(F\times_{\tau}B)\) by a new contraction, see \Cref{sec:twisted-shuffle} below.
The twisting cochain
\begin{equation}
  t\colon C(B)\to C(G)
\end{equation}
satisfies\footnote{%
  The sign in the twisting cochain condition is given incorrectly in~\cite[p.~29]{Shih:1962} and~\cite[Déf.~3.2.6]{Rubio:1991},
  but the same sign is correct in~\cite[p.~198]{Szczarba:1961} and~\cite[Def.~2.3]{Gugenheim:1972}. This difference is explained
  by the different orders of the factors in the twisted tensor product,
  compare~\cite[Def.~II.1.4]{HusemollerMooreStasheff:1974}.}%
\begin{equation}
  d(t) = t\cup t,
\end{equation}
and we have
\begin{equation}
  d_{t} = d_{\otimes} + (\mu\otimes1)(1\otimes t\otimes 1)(1\otimes\Delta)
\end{equation}
where \(\Delta\) is the diagonal of~\(C(B)\)
and the map~\(\mu\colon C(F)\otimes C(G)\to C(F)\) is induced by the \(G\)-action,
see~\cite[\S II.1, Thm.~2]{Shih:1962},~\cite[p.~410]{Gugenheim:1972}.
Conversely, \(t\) can be recovered from the differential~\(d_{t}\)
on the twisted tensor product~\(C(G)\otimes_{t}C(B)\) as the composition
\begin{equation}
  \label{eq:twc-from-twf}
  C(B) \xrightarrow{\eta\otimes1} C(G)\otimes C(B) \xrightarrow{d_{t}-d_{\otimes}} C(G)\otimes C(B) \xrightarrow{1\otimes\epsilon} C(G)
\end{equation}
where \(\epsilon\colon C(B)\to\kk\) is the augmentation and \(\eta\colon\kk\to C(G)\) the unit map \cite[p.~28]{Shih:1962}.

\subsection{Szczarba's twisting cochain}

Szczarba~\cite[Sec.~2]{Szczarba:1961} has defined an explicit twisting cochain\footnote{%
  Szczarba calls \(\phi(x)=-(-1)^{\deg{x}}\,t(x)\) a twisting cochain and that he does not use the Koszul sign convention,
  so that \((f\otimes g)(x\otimes y)=f(x)\,g(y)\) without the sign~\((-1)^{\deg{g}\deg{x}}\).}
\begin{equation}
  \tsz\colon C(B)\to C(G)
\end{equation}
and an explicit quasi-isomorphism~\(\psi\) between~\(C(B)\otimesL{\tsz} C(F)\) and~\(C(B\times_{\tau}F)\).
Note that Szczarba's ordering of the factors differs from Shih's; we will come back to this in \Cref{sec:twisted-shuffle}
where we also recall the definition of~\(\psi\).
In the proof below we are going to use a characterization of~\(\tsz\) due to Morace--Prouté~\cite[Sec.~6]{MoraceProute:1994}.

\begin{theorem}
  \label{thm:szarba-shih}
  Let \(F\times_{\tau}B\) be a twisted Cartesian product with structure group~\(G\).
  Then Szczarba's twisting cochain~\(\tsz\)
  is equal to Shih's twisting cochain based on the opposite Ei\-len\-berg--Zilber contraction~\((\AW,\shuffle,\hh)\).
\end{theorem}

\begin{proof}
  Assume first that \(B\) is reduced, that is, with a single vertex.
  In this case twisting functions~\(B_{>0}\to G\) correspond bijectively
  to maps of simplicial groups~\(q_{\tau}\colon\GG B\to G\), compare \cite[Cor.~27.2]{May:1968}.
  Both Szczarba's and Shih's twisting cochain are natural with respect to commutative diagrams
  \begin{equation}
    \label{eq:tw-natural}
    \begin{tikzcd}
      B_{>0} \arrow{d}{p} \arrow{r}{\tau} & G \arrow{d}{q} \\
      \tilde B_{>0} \arrow{r}{\tilde\tau} & \tilde G
    \end{tikzcd}
  \end{equation}
  in the sense that
  \begin{equation}
    \tilde t\,p_{*} = q_{*}\,t\colon C(B) \to C(\tilde G)
  \end{equation}
  holds for each of them. We may therefore assume \(G=\GG B\) in this case.
  
  Morace--Prouté have characterized Szczarba's cochain
  as the unique natural twisting cochain~\(t\colon C(B)\to C(\GG B)\) satisfying
  \begin{equation}
    \label{eq:norm-sz-tw}
    t(b) = \sigma(b)-1 \in (\GG B)_{0}
  \end{equation}
  for all~\(B\) (or just \(B=\bar\Delta[n]\)) and all~\(b\in B_{1}\) as well as
  \begin{equation}
    \label{eq:char-sz-tw}
    \bar h\,t(\bar e_{n}) = 0
  \end{equation}
  for all~\(n\ge2\), where \(\bar\Delta[n]\) is the \(n\)-simplex with all vertices identified,
  \(\bar e_{n}\in\bar\Delta[n]_{n}\) its fundamental simplex and \(\bar h\colon C(\bar\Delta[n])\to C(\bar\Delta[n])\)
  the contracting homotopy defined in~\cite[Sec.~3]{MoraceProute:1994}.
  As remarked in~\cite[p.~89]{MoraceProute:1994}, the condition~\eqref{eq:char-sz-tw} is satisfied if each simplex appearing in~\(t(\bar e_{n})\)
  is `right justified', meaning that it has the same final vertex~\(n\) as~\(\bar e_{n}\).
  
  Condition~\eqref{eq:norm-sz-tw} for Shih's twisting cochain follows from~\eqref{eq:twc-from-twf} and the identity
  \begin{align}
    (d_{t}-d_{\otimes})(1\otimes b) &= \AW\,\delta(1,b)
    = \AW\bigl( (\sigma(b),b_{0}) -(1,b_{0}) \bigr) \\
    \notag
    &= \bigl(\sigma(b)-1\bigr)\otimes b_{0}
  \end{align}
  for~\(b\in B_{1}\).
  If the last face operator is never used in the recursive definition of~\(t\),
  then all simplices appearing in the resulting chains are right-justified, so that \eqref{eq:char-sz-tw} holds.
  An inspection of the formulas for~\(\delta\) and~\(\hh\) shows that this is indeed the case,
  which proves our claim in case of a reduced base.

  For general~\(B\) we consider the reduced simplicial set~\(\tilde B\) obtained by identifying all vertices in~\(B\).
  The definition of a twisting function implies that \(\tau\) sends all degenerations of vertices to identity elements in~\(G\).
  It therefore induces a twisting function~\(\tilde\tau\colon\tilde B_{>0}\to G\) such that \eqref{eq:tw-natural} commutes for~\(q=\id_{G}\).
  Because Szczarba's and Shih's cochains agree for the twisted Cartesian product~\(F\times_{\tilde\tau}\tilde B\),
  they do so for~\(F\times_{\tau}B\). This completes the proof.
\end{proof}

\section{The twisted shuffle maps}
\label{sec:twisted-shuffle}

Shih's twisted differential~\eqref{eq:d-twisted-shih}
is part of a twisted Eilenberg--Zilber contraction, see~\cite[Thm.~II.1.1]{Shih:1962}.
The twisted shuffle map
\begin{equation}
  \label{eq:twisted-shuffle}
  \shuffle^{\tau}\colon C(F)\otimes_{\tsz}C(B) \to C(F\times_{\tau}B)
\end{equation}
is compatible with the right \(C(B)\)-comodule structure on both sides
and for~\(F=G\) also equivariant with respect to~\(C(G)\), see~\cite[Props.~II.4.2~\&~II.4.3]{Shih:1962}.
(The same holds in fact for the twisted Alexander--Whitney map and the twisted homotopy, compare~\cite[Lemma~4.5\(^{*}\)]{Gugenheim:1972}.)

On the other hand, Szczarba~\cite[Thm.~2.4]{Szczarba:1961} gives an explicit quasi-isomorphism
\begin{equation}
  \label{eq:quasi-sz}
  \psi\colon C(B)\otimesL{\tsz} C(F) \to C(B\times_{\tau}F).
\end{equation}
with the factors~\(B\) and~\(F\) swapped compared to Shih, hence with an adjusted definition
of the twisted differential on the left-hand side, \cf~\cite[Def.~II.1.4]{HusemollerMooreStasheff:1974}.
We observe that Szczarba's map enjoys the same nice properties as Shih's.

\begin{proposition}
  The map~\(\psi\) is a morphism of left \(C(B)\)-modules and, in the case~\(F=G\), also
  a morphism of right \(C(G)\)-modules.
\end{proposition}

Before entering the proof, let us recall that \(\psi\) is defined as the composition
\begin{multline}
  \qquad
  C(B)\otimesL{\tsz} C(F) \xrightarrow{\phi\otimes 1} C(B\times G)\otimes C(F) \\
  \xrightarrow{\shuffle_{B\times G,F}} C(B\times G\times F)
  \xrightarrow{(\id,\mu)_{*}} C(B\times_{\tau} F)
  \qquad
\end{multline}
where \(\mu\colon G\times F\to F\) is the action map, and
\begin{equation}
  \phi(b) = \sum_{\ii\in S_{n}} (-1)^{\deg{\ii}}\bigl(D_{0,\ii}^{n+1}b, D_{1,\ii}^{n+1}\sigma(b)\cdots D_{n,\ii}^{n+1}\sigma(b) \bigr)
  =\vcentcolon \sum_{\ii\in S_{n}} (-1)^{\deg{\ii}}\,\phi_{\ii}(b)
\end{equation}
for~\(b\in B_{n}\).\footnote{%
  In the definition of~\(\psi\) in~\cite[p.~201]{Szczarba:1961}
  the upper summation index should read ``\(p!\)''.}
Following Hess-Tonk's exposition~\cite[Sec.~1.4]{HessTonks:2006}, we write
\begin{gather}
  S_{n} = \bigl\{\, \ii=(i_{1},\dots,i_{n})\in\N^{n} \bigm| \text{\(0\le i_{s}\le n-s\) for all~\(s\)}\,\bigr\}
\shortintertext{and also}
  \deg{\ii} = \sum\ii =  i_{1}+\dots+i_{n}.
\end{gather}
The indices~\(i\) and~\(k\) Szczarba uses in the recursive definition of the simplicial operators~\(D_{j,i}^{n+1}=D_{j,\ii}^{n+1}\)
\cite[eq.~(3.1)]{Szczarba:1961} 
are related to~\(\ii\) via
\begin{equation}
  i = 1 + \sum_{s=1}^{n}\,(n-s)\,i_{s} \in \{\,1,\dots, n!\,\}
\end{equation}
and~\(k=i_{1}\). We moreover have \(\deg{\ii}+n=\epsilon(i,n+1)\) in Szczarba's notation, see the remark before~\cite[Thm.~7]{HessTonks:2006}.

\begin{proof}
  The equivariance with respect to the \(C(G)\)-action follows by naturality and
  associativity of the shuffle map. We now consider the \(C(B)\)-comodule structure.
  
  Because of naturality and the commutative diagram~\eqref{eq:shuffle-aw-2}
  it is enough to show that \(\phi\) is a morphism of left comodules,
  that is, to establish the identity
  \begin{equation}
    \partial_{k+1}^{n} b \otimes \phi(\partial_{0}^{k-1}b)
    = \sum_{\ii\in S_{n}} (-1)^{\deg{\ii}}\,\partial_{k+1}^{n}\,D_{0,\ii}^{n+1}\,b \otimes \partial_{0}^{k-1} \phi(b)
  \end{equation}
  for any~\(0\le k\le n\) and any~\(b\in B_{n}\).
  Recall that all operators~\(D_{j,\ii}^{n+1}\) are frontal \cite[p.~1866]{HessTonks:2006},
  hence satisfy the identity~\eqref{eq:frontal}.
  It follows by induction from the definition
  \begin{equation}
    D_{0,\ii}^{n+1} = \begin{cases}
      \bigl(D_{0,(i_{2},\dots,i_{n})}^{n+1}\bigr)' & \text{if \(i_{1}=0\),} \\
      \bigl(D_{0,(i_{2},\dots,i_{n})}^{n+1}\bigr)'s_{0}\,d_{i_{1}} = s_{0}\,D_{0,(i_{2},\dots,i_{n})}^{n+1}\,d_{i_{1}} & \text{if \(i_{1}>0\)}
    \end{cases}
  \end{equation}
  that \(\partial_{k+1}^{n}D_{0,\ii}^{n+1}\,b\) is degenerate unless \(i_{1}=\dots=i_{k}=0\).
  For~\(i_{1}=0\) the identity~%
  \( 
    \partial_{0}\,\phi_{\ii}(b) = \phi_{\ii}(\partial_{0}\,b)
  \) 
  holds, see~\cite[bottom of p.~205]{Szczarba:1961}. By induction this gives
  \begin{equation}
    \partial_{0}^{k-1} \phi_{\ii}(b) = \phi_{\ii}(\partial_{0}^{k-1}\,b)
  \end{equation}
  for~\(i_{1}=\dots=i_{k}=0\) and completes the proof.
\end{proof}

To arrive at a twisted tensor product of the form~\(C(B)\otimesL{\tilde t\,}C(F)\) with Shih's approach,
Gugenheim starts with a twisted Cartesian product~\(B\times_{\tilde\tau}F\)
where the twisting occurs in the \emph{last} face map, see~\cite[p.~406, Rem.]{Gugenheim:1972} and also~\cite[p.~53]{Rubio:1991}.
This leads to a \emph{different} twisting cochain~\(\tilde t\), essentially related to~\(\tsz\) via the passage
to opposite simplicial sets as in \Cref{sec:ez-reversed}.

There are two more ``canonical'' Eilenberg--Zilber contractions that use the transposed Alexander--Whitney map
\begin{gather}
  T_{C(Y),C(X)}\,\AW_{Y,X}\,(\tau_{X,Y})_{*}\colon C(X\times Y)\to C(X)\otimes C(Y), \\
  \notag (x,y) \mapsto \sum_{k=0}^{n} (-1)^{k(n-k)}\,\partial_{0}^{k-1}\,x \otimes \partial_{k+1}^{n}\,y
\end{gather}
for~\((x,y)\in (X\times Y)_{n}\),
together with the shuffle map and variants of~\(\h\) and~\(\hh\), see~\cite[p.~52]{Rubio:1991}.
Applying Shih's algorithm to one of them and the twisted Cartesian product~\(B\times_{\tau}F\)
does indeed lead to a differential induced by  Szczarba's twisting cochain~\(\tsz\) on the graded module~\(C(B)\otimes C(F)\).
However, because the Alexander--Whitney map has changed, so have the comodule structures over~\(C(B)\).
This means that one does not obtain Szczarba's twisted tensor product this way, either.
It therefore appears that \Cref{thm:szarba-shih} does not shed light onto
the precise relationship between the two twisted shuffle maps~\eqref{eq:twisted-shuffle} and~\eqref{eq:quasi-sz}
and the underlying twisted tensor products.

\end{document}